\documentclass[11pt]{amsart}
\usepackage{latexsym}
\usepackage{graphics}
\usepackage{graphicx}
\usepackage{epstopdf}
\usepackage{tikz}

\usepackage{caption}
\usepackage[labelformat=simple]{subcaption}

\newcommand{\Vol}{\mbox{\rm Vol}}

\newtheorem{theorem}{Theorem}[section]

\theoremstyle{remark}

 \theoremstyle{theorem}
  \newtheorem*{reptheorem}{Theorem \ref{Thm Augmentation}}
 \theoremstyle{theorem}
 \newtheorem*{reptheorem1}{Theorem \ref{Thm Upper Bound}}
 
\title[]{Simplicial volume of links from link diagrams}
\author{Oliver Dasbach \and Anastasiia Tsvietkova}
\date{}

\begin{document}

 \footnotesize
 \begin{abstract} The hyperbolic volume of a link complement is known to be unchanged when a half-twist is added to a link diagram, and a suitable 3-punctured sphere is present in the complement. We generalize this to the simplicial volume of link complements by analyzing the corresponding toroidal decompositions. We then use it to prove a refined upper bound for the volume in terms of twists of various lengths for links.  
\end{abstract}
 
\maketitle
\normalsize

\section{Introduction}

Determining the hyperbolic volume of a link complement in $S^3$ from a link diagram has been a topic of significant interest. Among the results of this nature are upper and lower bounds for the volume in terms of characteristics of a link diagram \cite{Adams:Thesis, FKP:Bounds, FKP:Guts, Purcell:TwistedTorus, Lackenby:Volume, Purcell:Twisted}, polynomials from link diagrams whose roots allow to compute hyperbolic structure and volume \cite{SakumaWeeks, Tsvietkova:2Bridge, TT:HyperbolicStructures}, and the results stating that, under the presence of a suitable embedded three-punctured sphere in the link complement, certain operations on link diagrams ``respect" the volume \cite{Adams:3punctured}. In particular, one can ``sum" link diagrams so that the hyperbolic volume of the links is additive, and one can add a crossing so that the volume of the link complement is unchanged.  In this note, we generalize the latter result by Colin Adams from hyperbolic to simplicial volume of links. We then use it to generalize the refined upper bound for the volume in terms of twists of various lengths \cite{DasbachTsvietkova} from alternating hyperbolic links to all links.

A \textit{full twist} is a bigon in a link diagram, and a \textit{half-twist} is a crossing. A \textit{crossing circle} is a link component that travels around two link strands near a crossing, as on Figure \ref{AugmentingA}. Adding a crossing circle to a diagram is called \textit{augmenting} the crossing. The proofs of several bounds for hyperbolic volume of links \cite{DasbachTsvietkova, Lackenby:Volume, Purcell:Twisted} are built on passing from an original link to a link for which a suitable polyhedral decomposition can be obtained. This is achieved by augmenting certain crossings and by removing adjacent full twists and half-twists. A half-twist may be removed without changing the volume due to Adams' result (\cite{Adams:3punctured}) stating that the hyperbolic volume is not changed under the operation depicted in Figure \ref{Augmenting}.
 
We generalize Adams' theorem as follows.

 \begin{reptheorem}  Let $M$ be a non-split link with a
 link diagram for which some part appears as in Figure \ref{AugmentingA}. Let $M'$ be the link obtained by
 replacing the tangle in the diagram of $M$ depicted in Figure \ref{AugmentingA} by the tangle
 in Figure \ref{AugmentingB}. Then the simplicial volume of $S^3 - M'$ equals the simplicial volume of $S^3- M$.
 \end{reptheorem}
 
 \begin{figure}
  \centering
  \begin{subfigure}[b]{0.4 \textwidth}
  \centering
  \includegraphics[scale=0.36]{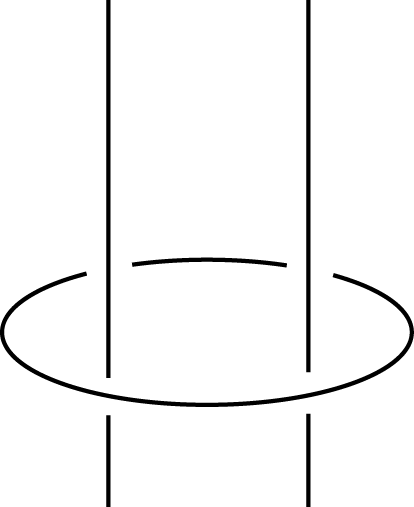}
  \caption{Before the replacement}
  \label{AugmentingA}
  \end{subfigure}
  \begin{subfigure}[b]{0.4 \textwidth}
  \centering
  \includegraphics[scale=0.36]{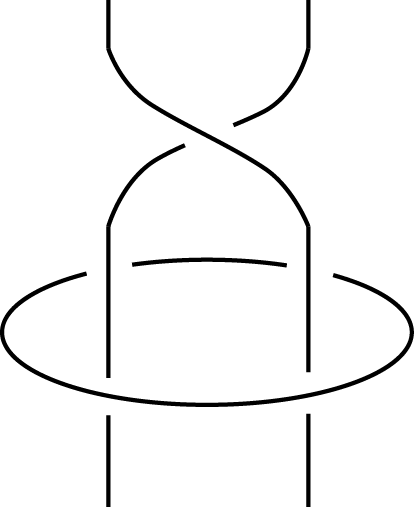} 
  \caption{After the replacement}
  \label{AugmentingB}
  \end{subfigure}
  \caption{The change of the link diagram that does not change the volume}
  \label{Augmenting}
  \end{figure}

As a corollary, the upper bounds for volume proved by Marc Lackenby, Ian Agol and Dylan Thurston \cite{Lackenby:Volume} generalize from hyperbolic volume to simplicial volume. However our original motivation comes from a new refined bound, which was originally established only for hyperbolic volume of \textit{alternating} links in \cite{DasbachTsvietkova} (and recently improved, using different techniques, by Colin Adams in \cite{Adams:Bipyramids}). To generalize this bound beyond alternating links, the use of simplicial volume is necessary, even under the assumption that the original link is hyperbolic.

In a link diagram, a \textit{twist region} is a sequence of bigons that is not a part of a larger sequence of bigons, or just a crossing. By $t_j$ we denote the number of twist regions with $j$ half-twists, and by $g_j$ the number of twist regions with at least $j$ half-twists.

 \begin{reptheorem1}
For a diagram $D$ of a non-split link $K$ the simplicial volume of $S^3-K$ is at most $10 g_4(D)+ 8t_3(D)+6t_2(D)+4 t_1(D) - a$,
  where $a =10$ if $g_4$ is non-zero, $a=7$ if $t_3$ is non-zero, and $a=6$ otherwise. In particular, if $K$ is a hyperbolic  knot, the hyperbolic volume
 $$ Vol(S^3- K) \leq (10 g_4(D)+ 8t_3(D)+6t_2(D)+4 t_1(D) - a ) v_3, $$
 where $v_3$ is the volume of a regular ideal hyperbolic tetrahedron. 
 \end{reptheorem1}
 

 \section{Adding a half-twist without changing the volume}

In this section, we generalize Adams' result \cite{Adams:3punctured} to the simplicial volume. For the background on simplicial volume, we refer the reader to \cite{Gromov:Volume}, and to Chapter 6 of \cite{Thurston:Book}.
 
 \begin{theorem} \label{Thm Augmentation} Let $M$ be a non-split link with a
  link diagram for which some part appears as in Figure \ref{AugmentingA}. Let $M'$ be the link obtained by
  replacing the tangle in the diagram of $M$ depicted in Figure \ref{AugmentingA} by the tangle
  in Figure \ref{AugmentingB}. Then the simplicial volume of $S^3 - M'$ equals the simplicial volume of $S^3- M$
 \end{theorem}

\begin{proof} Since $M$ is non-split, $M'$ is non-split as well, and both complements $S^3-M$, $S^3-M'$ are therefore irreducible.  Take the decomposition of $S^3-M$ by disjoint incompressible tori $T_1, T_2, ..., T_n$ embedded simultaneously in $S^3-M$, such that the simplicial volume of $S^3-M$ is the sum of hyperbolic volumes of its hyperbolic parts after decomposing along $T_1, T_2, ..., T_n$.

  \begin{figure}
   \centering
   \begin{subfigure}[b]{0.19 \textwidth}
   \centering
   \includegraphics[scale=0.48]{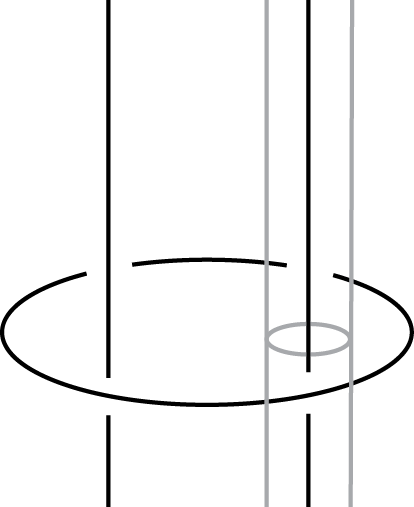}
   \caption{}
   \label{Torus1}
   \end{subfigure}
   \begin{subfigure}[b]{0.19 \textwidth}
   \centering
   \includegraphics[scale=0.48]{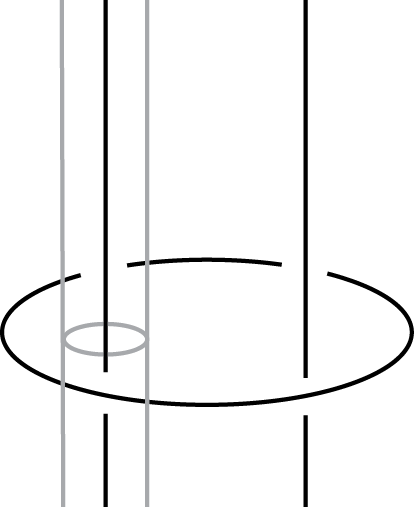} 
   \caption{}
   \label{Torus2}
   \end{subfigure}
   \begin{subfigure}[b]{0.19 \textwidth}
   \centering
   \includegraphics[scale=0.48]{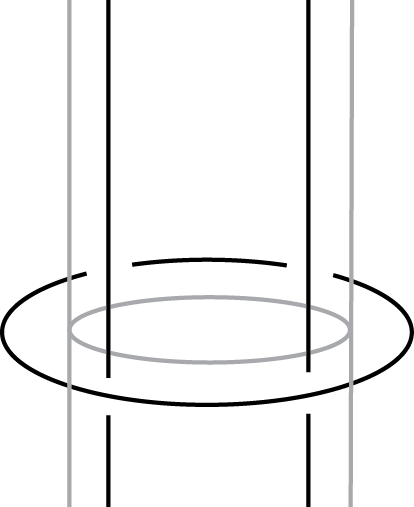} 
   \caption{}
   \label{Torus3}
   \end{subfigure}
   \begin{subfigure}[b]{0.19 \textwidth}
   \centering
   \includegraphics[scale=0.48]{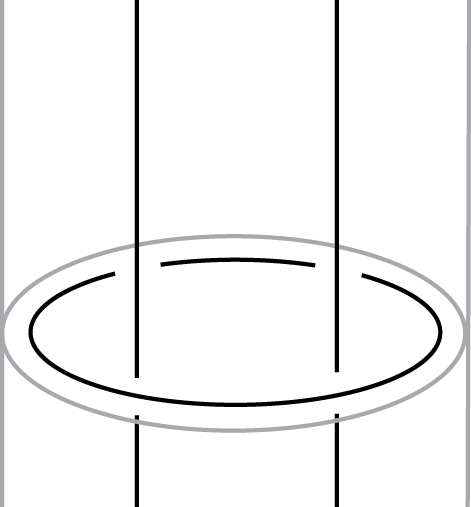} 
   \caption{}
   \label{Torus4}
   \end{subfigure}
    \begin{subfigure}[b]{0.19 \textwidth}
      \centering
      \includegraphics[scale=0.48]{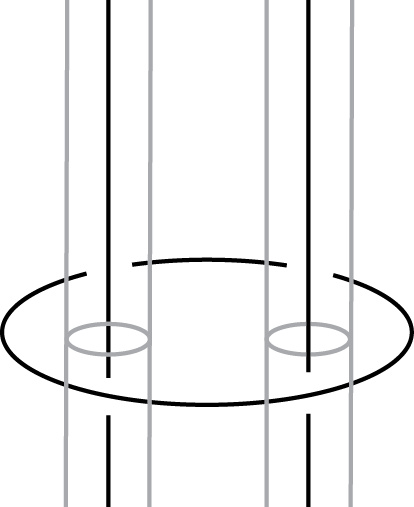} 
      \caption{}
      \label{Torus5}
      \end{subfigure}
   \caption{An incompressible torus intersecting $B_T$}
   \label{Torus}
   \end{figure}

    We will construct a toroidal decomposition of $S^3-M'$ that subdivides it into hyperbolic and Seifert-fibered pieces. Consider a ball $B\subset S^3$ that contains the tangle from Figure \ref{AugmentingA}.  The tangle is a properly embedded $1$-dimensional submanifold of $B$ meeting the boundary of $B$ transversely in four points. Denote the complement of the tangle in $B$ by $B_T$ (the boundary of $B_T$ is therefore a 4-punctured sphere). If a torus $T_i$ for some $i$ from 1 to $n$ can be isotoped so that it does not travel through $B_T$, then take the same torus $T_i$ for $S^3-M'$.

 Choose all the tori from $T_1, ..., T_n$ that cannot be isotoped away from $B_T$, and denote the collection by $\mathcal{T}$. Denote the crossing circle from Figure \ref{AugmentingA} by $C$. A $2$-punctured disk bounded by $C$ can be viewed as a $3$-punctured sphere which we denote by $S$. A meridian of every torus from $\mathcal{T}$ travels around a puncture of $S$ (otherwise, the torus is  compressible). Therefore, Figure \ref{Torus} shows all possible patterns of intersection of such a tori and $B_T$ (a torus of the decomposition is shown in grey). We will refer to them as to cases 1-5 respectively. 

 Among the tori in $\mathcal{T}$, choose an innermost one, say $T_j$ for some $j$ from $1$ to $n$. By {\it innermost} we mean that $T_j$ does not contain any other torus from the collection $\mathcal{T}$. An innermost torus is not necessarily unique, but the argument below applies to every innermost torus.

In cases 3 and 4, there is a $3$-punctured sphere inside $T_j$. Hence we can apply the original Adams' result to the piece of $S^3-M$ inside the torus $T_j$. This piece is hyperbolic if and only if the piece of $S^3-M$ inside the similar torus is hyperbolic \cite{Adams:3punctured}. If both pieces are hyperbolic, the volumes of the pieces are equal by \cite{Adams:3punctured}. Otherwise the pieces are Seifert-fibered, and do not contribute to the simplicial volumes of $S^3-M, S^3-M'$ respectively. Moreover, adding a half-twist takes place inside $T_j$ and does not affect any other pieces of the decomposition. Therefore, the simplicial volumes of $S^3-M$ and $S^3-M'$ are equal. 
 
In cases 1 and 2 we substitute the torus $T_j$ by a torus $T_j'$ when passing from $M$ to $M'$ as follows. Due to the symmetry we can consider just case 1. Cut $S^3-M$ and $T_j$ along the $3$-punctured sphere $S$, and re-glue the complement and the torus so as to add a half-twist as in Figure \ref{AugmentingB}. We obtain the complement of $M'$ and a new torus $T'_j$ as follows. The 2-tangle depicted on Figure \ref{AugmentingA} is connected with another $2$-tangle to form the link $M$. Denote this other $2$-tangle by $J$. Similarly, the 2-tangle depicted on Figure \ref{AugmentingB} is connected with $J$ to form the link $M'$. Let $T_j'$ be similar to $T_j$ in the tangle $J$, and parallel to the boundary of $M'$ outside $J$. In Figure \ref{Link} dotted grey lines show where $T_j'$ is different from $T_j$. Note that if we take all other tori for the decomposition of $S^3-M'$ the same as tori of the decomposition of $S^3-M$, then $T_j'$ does not intersect any other tori $T_k$. Otherwise either $T_k$ or $M$ would intersect $T_j$ for some $k=1,2,..., j-1,j+1,..., n$.
 
 \begin{figure}
   \centering
   \begin{subfigure}[b]{0.4 \textwidth}
   \centering
   \includegraphics[scale=0.5]{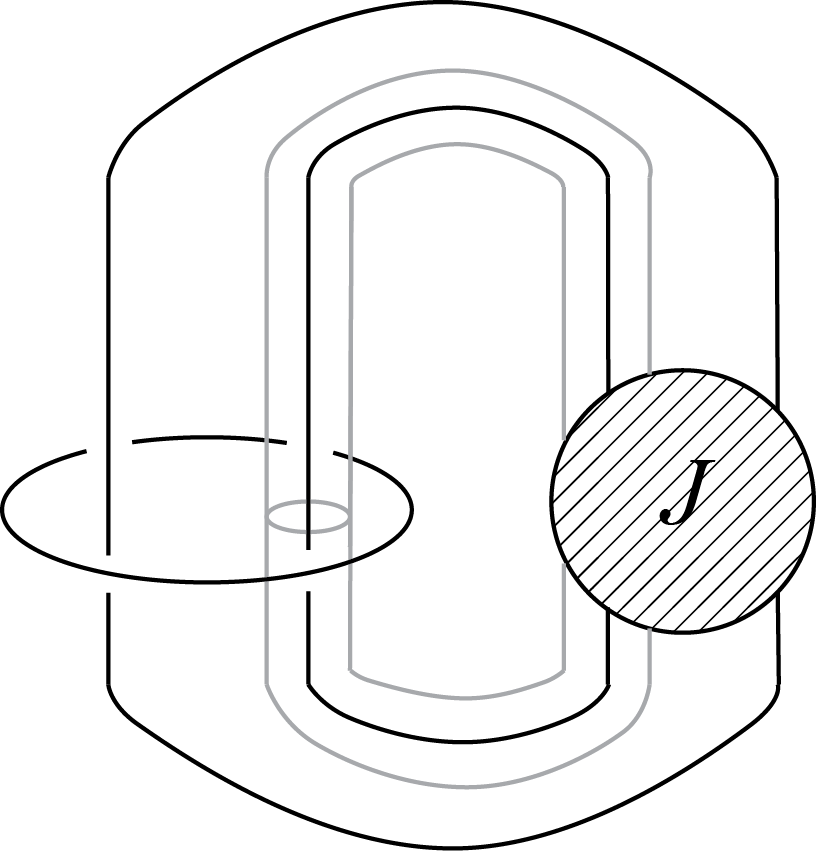}
   \caption{$T_j$ in $S^3-M$}
   \label{Link1}
   \end{subfigure}
   \begin{subfigure}[b]{0.4 \textwidth}
   \centering
   \includegraphics[scale=0.51]{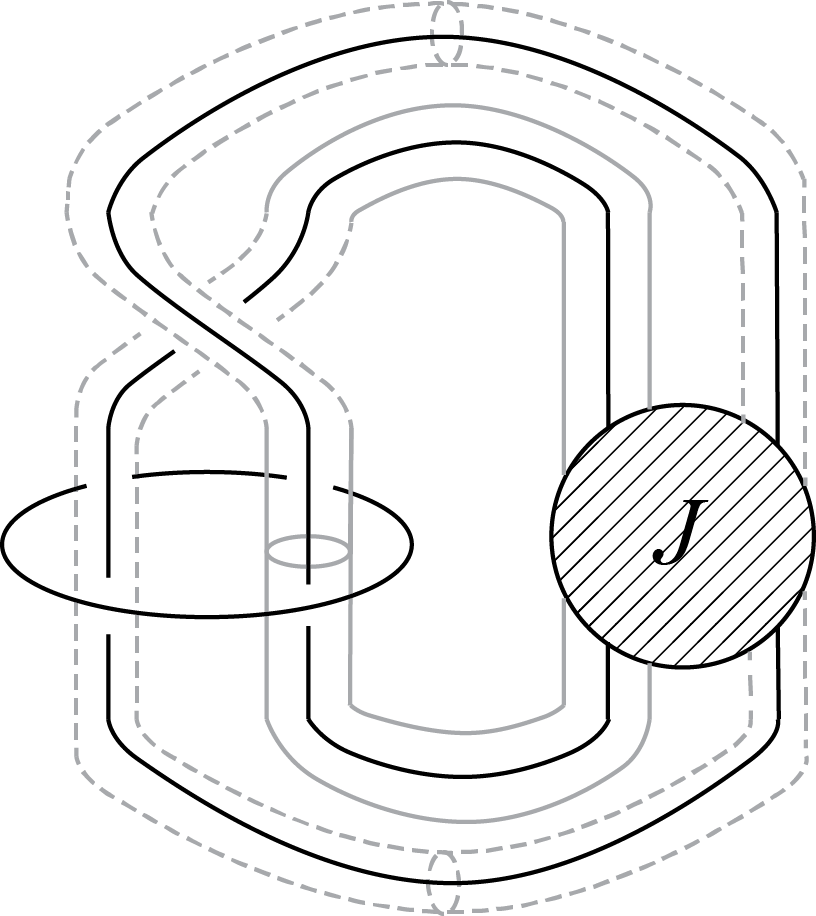} 
   \caption{$T_j'$ in $S^3-M'$}
   \label{Link2}
   \end{subfigure}
    \caption{}
     \label{Link}
     \end{figure}
  
  Let $J'$ be a subtangle of $J$ that is inside $T_j$. Note that inside $T_j$ and $T_j'$ we have homeomorphic pieces of the decomposition as shown in Figure \ref{InsideTorus}. On the figure, the decomposing torus, $T_j$ or $T_j'$, is in gray color, and components of the corresponding link, $M$ or $M'$, inside the torus, are in black. Hence, the interiors of $T_j$ and $T_j'$ contribute the same amounts to the simplicial volumes of $S^3-M$, and $S^3-M'$ respectively.
 
Now let us look at the pieces of the decompositions outside $T_j$ and $T_j'$. Note that $S^3-(M\cup T_j)$ and $S^3-(M'\cup T'_j)$ look the same except for the complement of the fragment from Figure \ref{Augmenting}. First suppose there are no other tori of the decomposition outside $T_j$ from the collection $\mathcal{T}$. Then we can use the original result from \cite{Adams:3punctured} to state that the volumes of the pieces of the decomposition outside $T_j$, $T_j'$ are equal, as we did above in cases 3 and 4. If there are other tori in $\mathcal{T}$, choose an innermost one in the sense that it possibly contains $T_j$, but does not contain any other tori from $\mathcal{T}$. Since both $T_j$ and $T_j'$ are boundary parallel in the tangle from Figure \ref{Augmenting}, we can again consider cases 1-5 for this new torus, and repeat the argument.

  \begin{figure}
    \centering
    \begin{subfigure}[b]{0.45 \textwidth}
    \centering
   \includegraphics[scale=0.4]{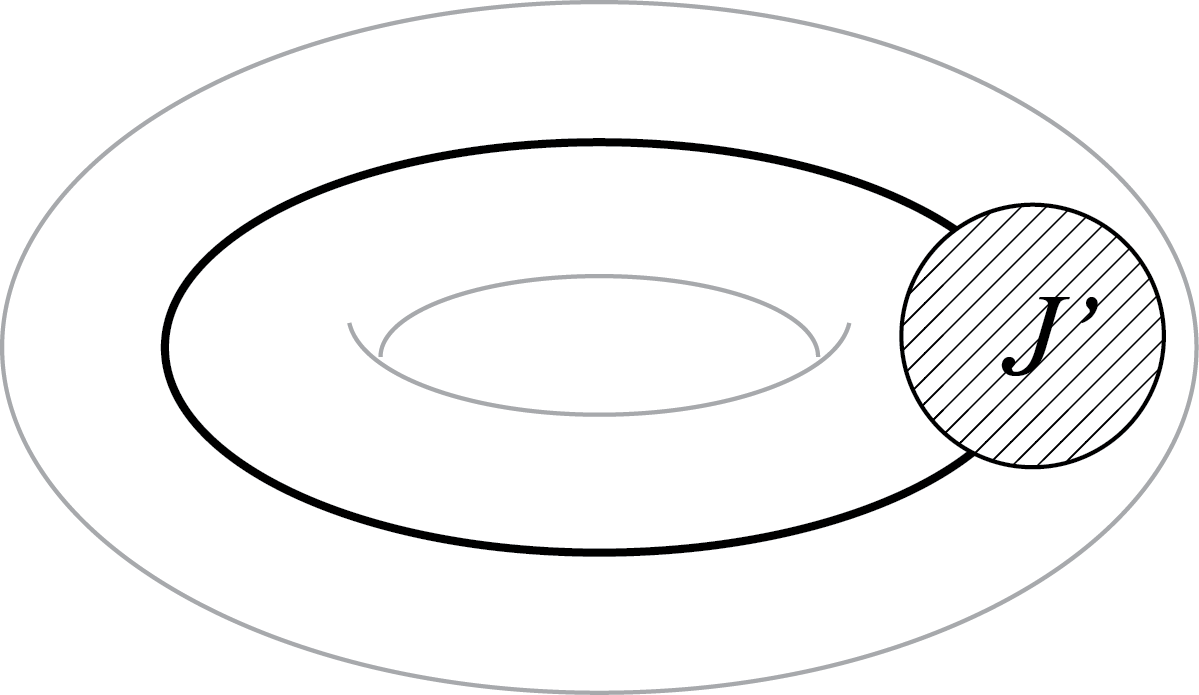}
   
    \caption{Inside the decomposing torus $T_i$ or $T_i'$}
     \label{InsideTorus}
    \end{subfigure}
    \begin{subfigure}[b]{0.45 \textwidth}
     \centering
     \includegraphics[scale=0.35]{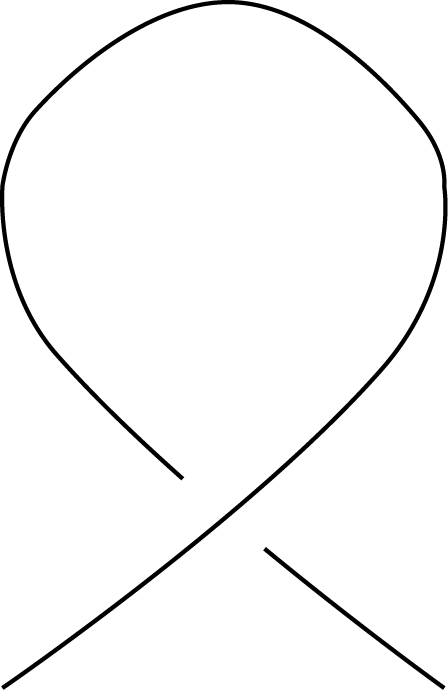}
      
      \caption{A loop in link diagram}
      \label{Loop}
      \end{subfigure}
      \caption{}
       \end{figure}

     In case 5, there is either the torus $T_j$ that intersects the sphere $S$ twice in $S^3-M$, or two (innermost) tori $T_j, T_k$ that intersect $S$ simultaneously as shown on Figure \ref{Torus5} in $S^3-M$. As in cases 1 and 2, cut $S^3-M$ and the torus/tori along the $3$-punctured sphere $S$, and re-glue the complement and the torus/tori so as to add a half-twist as in Figure \ref{AugmentingB}, obtaining $M'$. If we had one torus $T_j$ intersecting $S$ in $S^3-M$, we obtained either a new torus $T_j'$ or two new tori $T_j', T_j''$ instead. If we had two tori $T_j, T_k$ intersecting $S$ in $S^3-M$, they became one torus $T_j'$. The piece of the decomposition of $S^3-M'$ outside the new torus/tori contains a three-punctured sphere, and has the same volume as the the piece of $S^3-M$ outside the initial torus/tori by the original Adams' result. The pieces of $S^3-M$ and $S^3-M'$  inside these torus/tori are not necessarily homeomorphic, but we can decompose them further by essential tori into homeomorphic pieces as follows.  
     
     First, consider the case in which two tori $T_j, T_k$ decomposing $S^3-M$ become one torus $T_j'$ in $S^3-M'$. Since $T_j, T_k$ are not boundary parallel in $S^3-M$, each contains a tangle $J_1, J_2$ of $M$ respectively, and can be made boundary parallel elsewhere. In particular, the piece of the decomposition inside each of $T_j, T_j$ looks similar to the piece inside the torus on Figure \ref{InsideTorus}, but with $J_1$ or $J_2$ instead of $J'$ respectively. The piece of the decomposition inside the new torus $T_j'$ in $S^3-M'$ is then depicted on Figure \ref{Case5a} ($T_j'$ might be knotted in $S^3-M$, but this does not affect our reasoning). Add an essential torus $T_j''$ inside $T_j'$ that contains the tangle $J_1$ but is boundary parallel elsewhere. It decomposes the piece of $S^3-M'$ inside $T_j'$ into two pieces, and each piece is homeomorphic to the piece of $S^3-M$ inside the torus $T_j$ or $T_k$. Therefore, the simplicial volumes of $S^3-M, S^3-M'$ are equal.
     
     \begin{figure}
        \centering
        \begin{subfigure}[b]{0.4 \textwidth}
        \centering
        \includegraphics[scale=0.53]{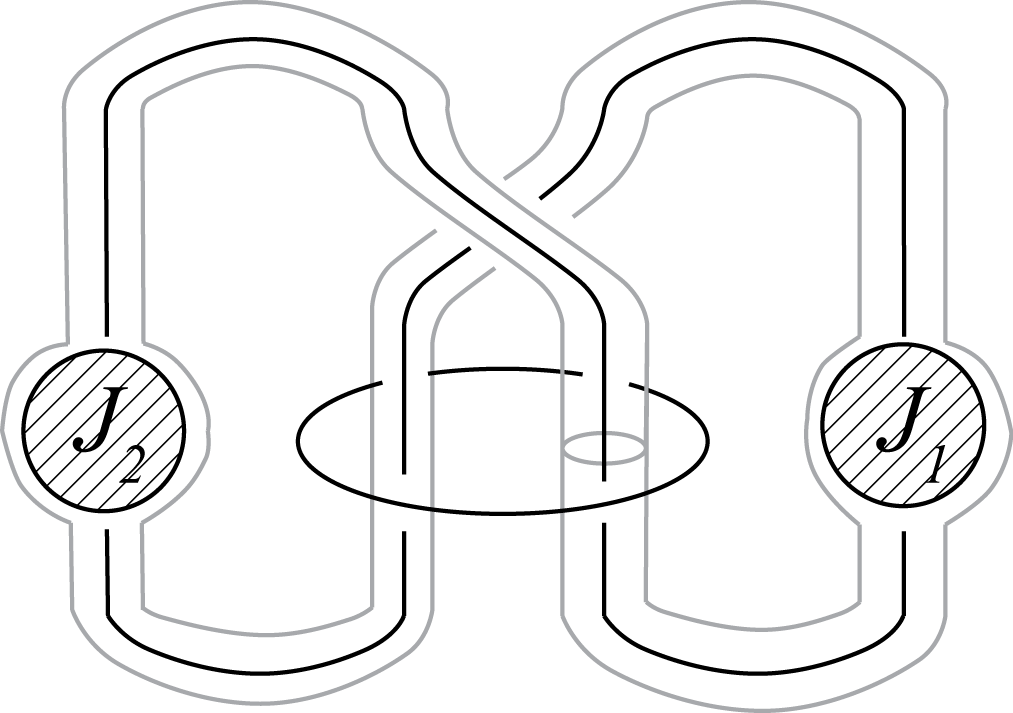}
        \caption{The torus $T'_j$.}
        \label{Case5a}
        \end{subfigure}
        \begin{subfigure}[b]{0.4 \textwidth}
        \centering
        \includegraphics[scale=0.52]{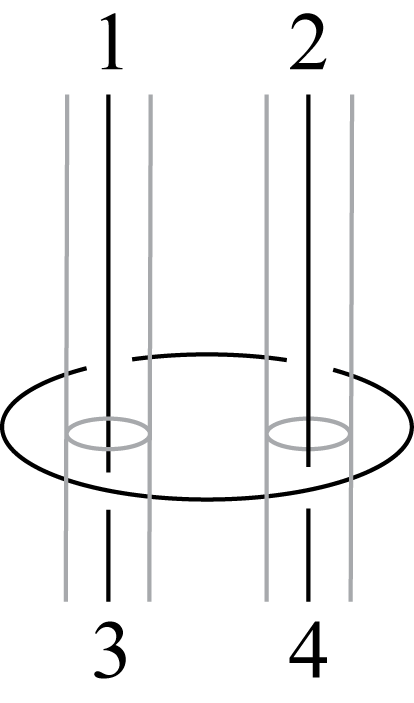} 
        \caption{}
        \label{Case5b}
        \end{subfigure}
         \caption{}
          \label{Case5}
          \end{figure}
     
     Now suppose there is just one torus $T_j$ depicted on Figure \ref{Torus5} in $S^3-M$, and it becomes two tori $T_j', T_k'$ in $S^3-M'$. The piece inside $T_j$ is then similar to the piece inside the torus in Figure \ref{Case5a}. Repeat the reasoning from the above paragraph, but now interchanging $M$ and $M'$ (\textit{i.e.} add an extra torus decomposing $S^3-M$ and lying inside $T_j$).
     
     Lastly, suppose we have one torus $T_j$ depicted on Figure \ref{Torus5} that becomes one torus $T_j'$ after re-gluing. Label the ends of the strands of the tangle of $M$ from Figure \ref{Torus5}  by 1, 2, 3, 4 respectively as in Figure \ref{Case5b}. For a single torus $T_j$ to become a single torus $T_j'$, 1 ought to be connected with 2, and 3 with 4 in $M$. The pieces of $M, M'$ inside the tori $T_j, T_j'$ respectively then consist of the tangles $J_1, J_2$ for $T_j$ and $J_1, J_2'$ from $T_j'$, where $J_2'$ is a mirror image of $J_2$.  And the piece of $S^3-M$ or $S^3-M'$ inside $T_j$ or $T_j'$ respectively is similar to the piece inside the torus in Figure \ref{Case5a} (with $J_2'$ instead of $J_2$ for $M'$). We can then decompose both $T_j, T_j'$ further into homeomorphic pieces as we did with one of the tori in two preceding paragraphs. For the pieces outside $T_j, T_j'$, if there are no other tori, apply the original Adams result. If there are other tori, choose an innermost one and repeat the reasoning.

          Repeat this argument until we have a torus that adheres to the cases 3 or 4, or until we exhaust the collection $\mathcal{T}$. This inductive argument proves that the simplicial volumes of $S^3-M$ and $S^3-M'$ are equal. 
 \end{proof}

\section{Refined upper bound for volume}

The work of Lackenby, and of Agol and Thurston \cite{Lackenby:Volume} shows that for a volume bound formulated in terms of the twist number $t$ of a link diagram, $10$ is an optimal constant, and the bound $10(t-1)v_3$ is sharp. In \cite{DasbachTsvietkova}, we refine this bound by involving more parameters into it. It provides an advantage in estimating volumes of many links, and it is used to show that the first and last three coefficients of the colored Jones polynomial correlate with the volume. The latter is an extension of results in \cite{DasbachLin:HeadAndTail, DL:VolumeIsh}. However, the refined bound was only proven for hyperbolic alternating links, and the argument does not carry over to other links. In this section, we provide a proof for all links.

Assume that the diagram of a link $K$ is reduced in the sense that there are no loops as shown in Figure \ref{Loop}, and that $K$ is not a split link. We start by modifying the link diagram through the addition of crossing circles. In particular, every twist region in the diagram that has at least four crossings (as, for example, in Figure \ref{Twist1}) is encircled by a crossing circle (as in Figure \ref{Twist2}). Denote the resulting link by $N$.   

\begin{figure}
    \centering
    \begin{subfigure}[b]{0.2 \textwidth}
    \centering
   \includegraphics[scale=0.45]{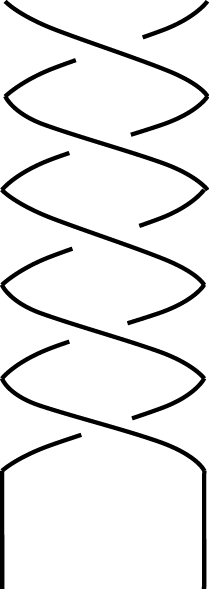}
   
    \caption{}
     \label{Twist1}
    \end{subfigure}
    \begin{subfigure}[b]{0.2 \textwidth}
     \centering
     \includegraphics[scale=0.45]{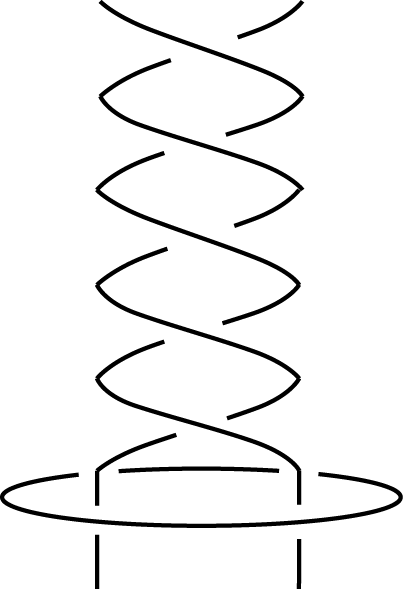}
      \caption{}
      \label{Twist2}
      \end{subfigure}
      \begin{subfigure}[b]{0.2 \textwidth}
           \centering
           \includegraphics[scale=0.45]{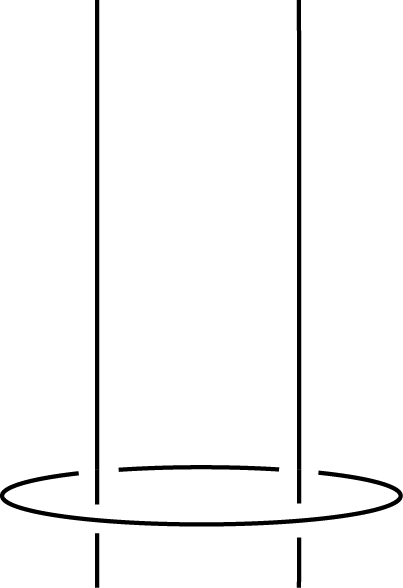}
            
            \caption{}
            \label{Twist3}
            \end{subfigure}
           \begin{subfigure}[b]{0.2 \textwidth}
                 \centering
                 \includegraphics[scale=0.3]{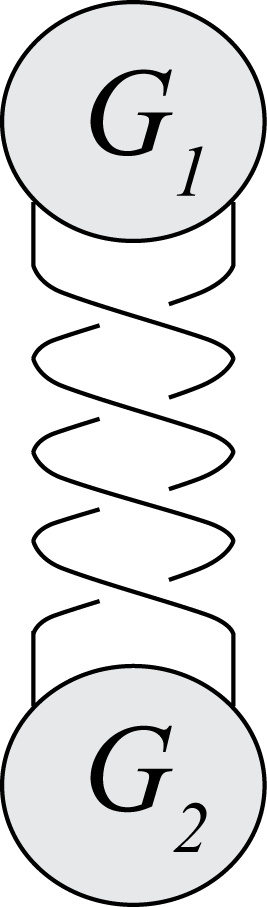}
                  
                  \caption{}
                  \label{KnotSum1}
                  \end{subfigure}
            \caption{}
       \end{figure}

We will briefly recall the argument from \cite{DasbachTsvietkova} for alternating links. If $K$ is a hyperbolic alternating link, then $N$ is a link that is called an augmented alternating link. Such links are hyperbolic by \cite{Adams:Augmented}. Moreover, $S^3 -K$ can be obtained from $S^3 - N$ by Dehn fillings of the tori that correspond to crossing circles. Therefore, the volumes satisfy $\Vol(S^3 - K)\leq\Vol(S^3 - N)$ by Theorem 6.5.6 in \cite{Thurston:Book} (the equality holds if and only if none of the twists is augmented, and the link is unchanged). Next delete all crossings in the twists that we encircled, obtaining a new link $L$. The volume of $S^3 - L$  is equal to the volume of $S^3 - N$ by Corollary 5 of \cite{Adams:3punctured}, and therefore $\Vol(S^3 - K)\leq\Vol(S^3 - L)$. 

 If $K$ is not alternating (and possibly not hyperbolic), $N$ might not be hyperbolic. Then consider the simplicial volume of $S^3-N$.  Proposition 6.5.2 from \cite{Thurston:Book} implies that the simplicial volume of the $S^3 - K$ is less than or equal to the simplicial volume of $S^3 - N$. We apply the proposition to the manifold $M$ that is $S^3-N$ with boundary, and together with the necessary number of solid tori that have zero simplicial volume. 
 
  Now delete all crossings in the encircled twists, obtaining $L$ (as in Figure \ref{Twist3}). If $N$ is non-split, then Theorem \ref{Thm Augmentation} implies that the simplicial volume of 
 $S^3 - L$  is equal to the simplicial volume of $S^3 - N$, and therefore the simplicial volume of  $S^3-K$ is less or equal to the simplicial volume of $S^3-L$. If $K$ is, in addition, hyperbolic, then the hyperbolic volume $\Vol(S^3-K)$ is less than or equal to the simplicial volume of $S^3-N$  by Lemma 6.5.4 in \cite{Thurston:Book}. 
 
 Suppose $N$ is split. Since $K$ is not a split link, the only components of $N$ that can be separated by $2$-spheres embedded in $S^3-N$ are the introduced crossing circles. A crossing circle can be separated by a $2$-sphere from the rest of the diagram only if it was introduced at a nugatory twist $c$ of $K$. Suppose $c$ connects $K$ into two 2-tangles, $G_1$ and $G_2$, as on Figure \ref{KnotSum1}. Then we can manipulate the diagram of $K$ so that there is no more $c$, and the diagram of $K$ is a knot sum of two tangles, either $G_1$ and $G_2$, or $G_1$ and the mirror image of $G_2$. Once this is done for every nugatory twist, augment the resulting diagram instead of augmenting the initial diagram of $K$. The resulting link $N$ is non-split, and the above argument applies.  
 
 We therefore obtain the following.

\begin{theorem}\label{Thm Upper Bound} Given a diagram $D$ of a non-split link $K$, with $t_i$ twist regions of precisely $i$ crossings, and $g_i$ twist regions of at least $i$ crossings, the simplicial volume of $S^3-K$ is at most $10 g_4(D)+ 8t_3(D)+6t_2(D)+4 t_1(D) - a$,
  where $a =10$ if $g_4$ is non-zero, $a=7$ if $t_3$ is non-zero, and $a=6$ otherwise. In particular, if $K$ is a hyperbolic  knot, the hyperbolic volume
 $$ Vol(S^3- K) \leq (10 g_4(D)+ 8t_3(D)+6t_2(D)+4 t_1(D) - a ) v_3, $$
 where $v_3$ is the volume of a regular ideal hyperbolic tetrahedron. 
\end{theorem}

\section{Acknowledgments}
 
We are thankful to the anonymous referee of \cite{DasbachTsvietkova}, who suggested to use simplicial volume in order to generalize the upper bound. Moreover, we are grateful to the referee of the current paper for his careful review. 
We acknowledge the support from U.S. National Science Foundation grants DMS-1317942 and DMS-1406588.

\bibliography{WithAnastasiia}
\bibliographystyle {amsalpha}

\parbox[t]{2.8in}{
Oliver Dasbach\\
Department of Mathematics\\
Louisiana State University\\
Baton Rouge, Louisiana 70803, USA.\\
kasten@math.lsu.edu}
\qquad
\parbox[t]{2.8in}{
Anastasiia Tsvietkova\\
Department of Mathematics\\
University of California, Davis \\
One Shields Ave, Davis, CA 95616\\
tsvietkova@math.ucdavis.edu
}

\end{document}